\documentclass{amsart}

\usepackage{amssymb}
\usepackage{hyperref}
\hypersetup{colorlinks=true, citecolor=blue}

\newtheorem{theorem}{Theorem}[section]
\newtheorem{lemma}[theorem]{Lemma}
\newtheorem{proposition}[theorem]{Proposition}
\newtheorem{corollary}[theorem]{Corollary}

\theoremstyle{definition}
\newtheorem{definition}[theorem]{Definition}
\theoremstyle{remark}
\newtheorem{remark}[theorem]{Remark}

\begin{document}

\title{Harmonic Bloch Space on the Real Hyperbolic Ball}

\thanks{This research is supported by Eski\c{s}ehir Technical University
Research Fund under grant 23ADP054.}

\author{A. Ers$\dot{\hbox{\i}}$n \"Ureyen}
\address{Department of Mathematics, Faculty of Science,
Eski\c{s}ehir Technical University, 26470, Eski\c{s}ehir,
Turkey}
\email{aeureyen@eskisehir.edu.tr}

\date{\today}



\begin{abstract}
We study the Bloch and the little Bloch spaces of harmonic
functions on the real hyperbolic ball.
We show that the Bergman projections from $L^\infty(\mathbb B)$ to
$\mathcal B$, and from $C_0(\mathbb B)$ to $\mathcal B_0$ are onto.
We verify that the dual space of the hyperbolic harmonic Bergman
space $\mathcal B^1_\alpha$ is $\mathcal B$ and its predual
is $\mathcal B_0$.
Finally, we obtain an atomic decomposition of Bloch functions
as a series of Bergman reproducing kernels.
\end{abstract}

\maketitle

\section{Introduction}

For $n\geq 2$ and $x,y\in\mathbb R^n$, let
$\langle x,y\rangle=x_1y_1+\dots+x_ny_n$ be the Euclidean inner
product and $\lvert x\rvert=\sqrt{\langle x,x\rangle}$ the
corresponding norm.
Let $\mathbb B=\mathbb B_n=\{x\in\mathbb R^n : \lvert x\rvert <1\}$
be the unit ball and $\mathbb S=\partial\mathbb B$ the unit sphere.

The hyperbolic ball is $\mathbb B$ equipped with the hyperbolic metric
\begin{equation*}
ds^2=\frac{4}{(1-\lvert x\rvert^2)^2}\sum_{i=1}^n dx_i^2.
\end{equation*}
For a $C^2$ function $f$, the hyperbolic (invariant) Laplacian
$\Delta_h$ is defined by
\begin{equation*}
\Delta_h f(a)=\Delta(f\circ\varphi_a)(0)
\qquad (a\in\mathbb B),
\end{equation*}
where $\Delta=\partial^2/\partial x_1^2+\dots+\partial^2/\partial x_n^2$
is the Euclidean Laplacian and $\varphi_a$ is the involutory M\"obius
transformation given in \eqref{definevarphia} that exchanges $a$ and $0$.
Up to a factor $1/4$, $\Delta_h$ is the Laplace-Beltrami operator
associated with the hyperbolic metric.
A straightforward calculation shows
\begin{equation*}
\Delta_hf(a)=(1-\lvert a\rvert^2)^2\Delta f(a)
+2(n-2)(1-\lvert a\rvert^2)\langle a,\nabla f(a)\rangle,
\end{equation*}
where
$\nabla=\big(\partial/\partial x_1,\dots,\partial/\partial x_n\big)$
is the Euclidean gradient.
We refer the reader to \cite[Chapter 3]{St1} for details.

A $C^2$ function $f\colon\mathbb B\to\mathbb C$ is called
hyperbolic harmonic or $\mathcal H$-harmonic on $\mathbb B$
if $\Delta_h f(x)=0$ for all $x\in\mathbb B$.
We denote by $\mathcal H(\mathbb B)$ the space of all
$\mathcal H$-harmonic functions equipped with the topology of
uniform convergence on compact subsets.

Let $\nu$ be the Lebesgue measure on $\mathbb B$ normalized so that
$\nu(\mathbb B)=1$ and for $\alpha>-1$, let
$d\nu_\alpha(x)=(1-\lvert x\rvert^2)^\alpha d\nu(x)$.
For $0<p<\infty$, denote the Lebesgue classes with respect to
$d\nu_\alpha$ by $L^p_\alpha(\mathbb B)$.
The $\mathcal H$-harmonic weighted Bergman space
$\mathcal B^p_\alpha$ is the subspace
$L^p_\alpha(\mathbb B)\cap\mathcal H(\mathbb B)$.
When $p=2$, $\mathcal B^2_\alpha$ is a reproducing kernel Hilbert space
and for each $x\in\mathbb B$, there exists
$\mathcal R_\alpha(x,\cdot)\in\mathcal B^2_\alpha$ such that
\begin{equation}\label{Reproduce}
f(x)=\int_{\mathbb B} f(y)\overline{\mathcal R_\alpha(x,y)}\,d\nu_\alpha(y)
\qquad (f\in \mathcal B^2_\alpha).
\end{equation}
The reproducing kernel $\mathcal R_\alpha$ is real-valued and the
conjugation above can be deleted.
$\mathcal R_\alpha(x,y)=\mathcal R_\alpha(y,x)$ and so $\mathcal R_\alpha$
is $\mathcal H$-harmonic as a function of each variable.
We refer the reader to \cite{Sou} and \cite[Chapter 10]{St1} for details.

For $\alpha>-1$ and $\phi\in L^1_\alpha$, the Bergman projection
operator $P_\alpha$ is defined by
\begin{equation*}
P_\alpha\phi(x)=\int_{\mathbb B} \mathcal R_\alpha(x,y)\phi(y)\,d\nu_\alpha(y).
\end{equation*}
In \cite{U1}, estimates for the reproducing kernels have been obtained,
and it is shown that $P_\gamma:L^p_\alpha\to\mathcal B^p_\alpha$ is
bounded if and only $\alpha+1<p(\gamma+1)$.
Further properties of the spaces $\mathcal B^p_\alpha$ including their
atomic decomposition have been obtained in \cite{U2}.

The purpose of this paper is to consider the $p=\infty$, i.e., the
Bloch space, case.
For a $C^1$ function $f$, the hyperbolic gradient $\nabla^h$ is defined by
\begin{equation*}
\nabla^h f(a)=-\nabla(f\circ\varphi_a)(0)=(1-\lvert a\rvert^2)\nabla f(a).
\end{equation*}
The $\mathcal H$-harmonic Bloch space $\mathcal B$ consists of all
$f\in\mathcal H(\mathbb B)$ such that
\begin{equation}\label{defineBloch}
p_{\mathcal B}(f)
=\sup_{x\in\mathbb B}\lvert\nabla^h f(x)\rvert
=\sup_{x\in\mathbb B}(1-\lvert x\rvert^2)\lvert\nabla f(x)\rvert<\infty.
\end{equation}
$p_{\mathcal B}$ is a seminorm and
$\|f\|_{\mathcal B}=\lvert f(0)\rvert+p_{\mathcal B}(f)$ is a norm
on $\mathcal B$.
The little Bloch space $\mathcal B_0$ is the subspace
consisting of functions $f$ with
$\lim_{\lvert x\rvert\to 1^-}
\lvert(1-\lvert x\rvert^2)\lvert\nabla f(x)\rvert=0$.

The properties we state below for the $\mathcal H$-harmonic Bloch space
$\mathcal B$ are similar to the holomorphic or the harmonic case.
However, we would like to point out that there are differences between
these and the $\mathcal H$-harmonic case.
For example, it is well known that polynomials are dense in the holomorphic
little Bloch space and similarly harmonic polynomials are dense
in the harmonic little Bloch space.
However, this is not true in the $\mathcal H$-harmonic case.
In fact, when the dimension $n$ is odd, there are not any non-constant
$\mathcal H$-harmonic polynomials (see Lemma \ref{Lpolyhar}).
Besides, some elementary properties of harmonic (or holomorphic)
functions do not hold for $\mathcal H$-harmonic functions.
For example, if $f$ is harmonic, then the partial derivative
$\partial f/\partial x_i$ and the dilation $f_r(x)=f(rx)$
are also harmonic.
However, neither of these are true for $\mathcal H$-harmonic functions.
Therefore, even if the final results are similar, some of the proofs
in the harmonic (or holomorphic) case do not directly carry over to
the $\mathcal H$-harmonic case.

Our first result is about projections onto $\mathcal B$ and $\mathcal B_0$.
Let $L^\infty(\mathbb B)$ be the Lebesgue space of essentially bounded
functions, $C(\overline{\mathbb B})$ be the space of functions continuous
on $\overline{\mathbb B}$, and $C_0(\mathbb B)$ be its subspace consisting
of functions vanishing on $\partial\mathbb B$.

\begin{theorem}\label{Tproj}
For every $\alpha>-1$, $P_\alpha$ maps $L^\infty(\mathbb B)$ boundedly
onto $\mathcal B$.
It also maps $C(\overline{\mathbb B})$ and $C_0(\mathbb B)$ boundedly
onto $\mathcal B_0$.
\end{theorem}

It has already been verified in \cite[Theorem 1.5]{U1} that
$P_\alpha\colon L^\infty(\mathbb B)\to\mathcal B$ is bounded and the main
aspect of the above theorem is the surjectivity.
To achieve this we first characterize $\mathcal B$ and $\mathcal B_0$ in
terms of certain fractional differential operators that are defined in
Section \ref{SDiffOp}.
These operators are compatible with $\mathcal H$-harmonic functions and
the reproducing kernels, and to understand the properties of $\mathcal B$,
they are more adequate than $\nabla^h$ or $\nabla$ used
in \eqref{defineBloch}.

We next consider the duality problem.
For $1<p<\infty$, the dual of the hyperbolic Bergman space
$\mathcal B^p_\alpha$ can be identified with $\mathcal B^{p'}_\alpha$,
where $p'=p/(p-1)$ is the conjugate exponent of $p$
(see \cite[Corollary 1.4]{U1}).
We complete the missing $p=1$ case.

\begin{theorem}\label{Tdual}
For $\alpha>-1$, the dual of $\mathcal B^1_\alpha$ can be
identified with $\mathcal B$ under the pairing
\begin{equation}\label{pairing}
\langle f,g\rangle_\alpha
=\lim_{r\to 1^-}\int_{r\mathbb B}f(x)g(x)\,d\nu_\alpha(x).
\end{equation}
More precisely, to each $\Lambda\in(\mathcal B^1_\alpha)^*$,
there corresponds a unique $g\in\mathcal B$ with
$\|g\|_{\mathcal B}$ equivalent to $\|\Lambda\|$ such that
$\Lambda(f)=\langle f,g\rangle_\alpha$.
Similarly, for every $\alpha>-1$ the dual of $\mathcal B_0$
can be identified with $\mathcal B^1_\alpha$ under the
pairing \eqref{pairing}.
\end{theorem}

For an unbounded $g\in\mathcal B$ and $f\in\mathcal B^1_\alpha$, the
integral $\int_{\mathbb B}f(x)g(x)d\nu_\alpha(x)$ may not absolutely
converge; however the limit in \eqref{pairing} always exists.
In the case of the \textit{holomorphic} Bloch space on the unit
ball of $\mathbb C^n$, $g(z)=\log 1/(1-z_1)$ is an unbounded
Bloch function.
On the other hand, it is not obvious whether there exists an unbounded
$\mathcal H$-harmonic Bloch function.
We give an example in Lemma \ref{Lunbdd}.

As our final result we prove atomic decomposition of
$\mathcal H$-harmonic Bloch functions.
Atomic decomposition of \textit{harmonic} Bergman and Bloch functions
have been obtained in \cite{CR} (see also \cite{CL}).
In the $\mathcal H$-harmonic case, see \cite{Ja2} for atomic
decomposition of Hardy spaces and \cite{U2} for Bergman spaces.

For $a,b\in\mathbb B$, the pseudo-hyperbolic metric is
$\rho(a,b)=\lvert\varphi_a(b)\rvert$.
For $0<r<1$, let $E_r(a)=\{\,x\in\mathbb B:\rho(x,a)<r\,\}$ be the
pseudo-hyperbolic ball with center $a$ and radius $r$.
A sequence $\{a_m\}$ of points of $\mathbb B$ is called $r$-separated
if $\rho(a_m,a_k)\geq r$ for $m\neq k$.
An $r$-separated sequence is called an $r$-lattice if
$\bigcup_{m=1}^\infty E_r(a_m)=\mathbb B$.
Denote by $\ell^\infty$, the space of bounded sequences with norm
$\|\{\lambda_m\}\|_{\ell^\infty}=\sup_{m\geq 1}\lvert\lambda_m\rvert$,
and by $c_0$ the subspace consisting of sequences that converge to $0$.

\begin{theorem}\label{TAtomic}
Let $\alpha>-1$.
There is an $r_0<1/2$ depending only on $n$ and $\alpha$ such that
if $\{a_m\}$ is an $r$-lattice with $r<r_0$, then for every
$f\in\mathcal B$ (resp. $\mathcal B_0$), there exits
$\{\lambda_m\}\in\ell^\infty$ (resp. $c_0$) such that
\begin{equation}\label{atomicrepr}
f(x)=\sum_{m=1}^\infty\lambda_m
\frac{\mathcal R_\alpha(x,a_m)}
{\|\mathcal R_\alpha(\cdot,a_m)\|_{\mathcal B}}
\qquad (x\in\mathbb B),
\end{equation}
where the series converges absolutely and uniformly on compact
subsets of $\mathbb B$ and the norm $\|\{\lambda_m\}\|_{\ell^\infty}$
is equivalent to the norm $\|f\|_{\mathcal B}$.
\end{theorem}

By Lemma \ref{LkernelinB}, the norm
$\|\mathcal R_\alpha(\cdot,a_m)\|_{\mathcal B}$ is equivalent
to $(1-\lvert a_m\rvert^2)^{-(\alpha+n)}$ and the theorem remains true
if one uses the representation
\begin{equation}\label{Repr2}
f(x)=\sum_{m=1}^\infty\lambda_m(1-\lvert a_m\rvert^2)^{\alpha+n}
\,\mathcal R_\alpha(x,a_m)
\qquad (x\in\mathbb B)
\end{equation}
instead of \eqref{atomicrepr}.

\section{Preliminaries}

We denote positive constants whose exact values are inessential by
the letter $C$.
The value of $C$ may be different in different occurrences.
For two positive expressions $X$ and $Y$, we write $X\lesssim Y$
to mean $X\leq C Y$.
If both $X\leq CY$ and $Y\leq CX$, we write $X\sim Y$.

For $x,y\in\mathbb B$, we define
\begin{equation*}
[x,y]:=\sqrt{1-2\langle x,y\rangle+\lvert x\rvert^2\lvert y\rvert^2}.
\end{equation*}
$[x,y]$ is symmetric, $[x,0]=1$, and if $y\neq 0$, then
$[x,y]=\bigl\lvert \lvert y\rvert x-y/\lvert y\rvert\bigr\rvert$.
Therefore
\begin{equation}\label{xybig}
[x,y]\geq 1-\lvert x\rvert\lvert y\rvert
\qquad (x,y\in\mathbb B).
\end{equation}

Denote by $\mathcal M(\mathbb B)$ the group of M\"obius
transformations that preserve $\mathbb B$.
For $a\in\mathbb B$, the canonical M\"obius transformation
that exchanges $a$ and $0$ is given by
\begin{equation}\label{definevarphia}
\varphi_a(x)
=\frac{a\lvert x-a\rvert^2+(1-\lvert a\rvert^2)(a-x)}{[x,a]^2}
\qquad (x\in\mathbb B).
\end{equation}
It is an involution, $\varphi_a^{-1}=\varphi_a$, and
for all $x\in\mathbb B$, the identity
\begin{equation}\label{MobiusIdnt}
1-\lvert\varphi_a(x)\rvert^2
=\frac{(1-\lvert a\rvert^2)(1-\lvert x\rvert^2)}{[x,a]^2}
\end{equation}
holds.
The determinant of the Jacobian matrix of $\varphi_a$ satisfies
(\cite[Theorem 3.3.1]{St1})
\begin{equation}\label{Jacob}
\lvert \det J \varphi_a(x)\rvert
=\frac{(1-\lvert\varphi_a(x)\rvert^2)^n}{(1-\lvert x\rvert^2)^n}.
\end{equation}
The equality
\begin{equation}\label{avarphia}
[a,\varphi_a(x)]=\frac{1-\lvert a\rvert^2}{[x,a]}
\end{equation}
follows from \eqref{MobiusIdnt}
(see \cite[Theorem 1.1]{RK} or \cite[Lemma 2.1]{U2}).

For $a,b\in\mathbb B$, the pseudo-hyperbolic metric
$\rho(a,b)=\lvert\varphi_a(b)\rvert$ satisfies the
equality
\begin{equation}\label{phmetric}
\rho(a,b)=\frac{\lvert a-b\rvert}{[a,b]}.
\end{equation}
The pseudo-hyperbolic ball
$E_r(a)=\{\,x\in\mathbb B:\rho(x,a)<r\,\}$ is also a Euclidean ball
with (see \cite[Theorem 2.2.2]{St1})
\begin{equation}\label{Ercenter}
  \text{center}=\frac{(1-r^2)a}{1-\lvert a\rvert^2r^2}
\quad \text{and}\quad
\text{radius}=\frac{(1-\lvert a\rvert^2)r}{1-\lvert a\rvert^2r^2}.
\end{equation}

For a proof of the following lemma,
see \cite[Lemma 2.1 and 2.2]{CKL}.
\begin{lemma}\label{Lratiobracket}
(i) For all $a,b\in\mathbb B$,
\begin{equation*}
\frac{1-\rho(a,b)}{1+\rho(a,b)}
\leq\frac{1-\lvert a\rvert}{1-\lvert b\rvert}
\leq\frac{1+\rho(a,b)}{1-\rho(a,b)}.
\end{equation*}
(ii) For all $a,b,x\in\mathbb B$,
\begin{equation*}
\frac{1-\rho(a,b)}{1+\rho(a,b)}
\leq\frac{[x,a]}{[x,b]}
\leq\frac{1+\rho(a,b)}{1-\rho(a,b)}.
\end{equation*}
\end{lemma}
The hyperbolic metric on $\mathbb B$ is given by
\begin{equation*}
\beta(a,b)
=\log\frac{1+\rho(a,b)}{1-\rho(a,b)}
\qquad (a,b\in\mathbb B).
\end{equation*}
Both metrics $\rho$ and $\beta$ are M\"obius invariant.

Finally, we mention two integral estimates.
Let $\sigma$ be the normalized surface measure on $\mathbb S$.
For a proof of the following lemma see \cite[Proposition 2.2]{LS}.

\begin{lemma}\label{LIntonSandB}
Let $s>-1$ and $t\in\mathbb R$.
For all $x\in\mathbb B$,
\begin{equation*}
\int_{\mathbb S}\frac{d\sigma(\zeta)}{\lvert x-\zeta\rvert^{n-1+t}}
\sim\int_{\mathbb B}\frac{(1-\lvert y\rvert^2)^s}{[x,y]^{n+s+t}}
d\nu(y)
\sim
\begin{cases}
\dfrac{1}{(1-\lvert x\rvert^2)^{t}},&\text{if $t>0$};\\
1+\log\dfrac{1}{1-\lvert x\rvert^2},&\text{if $t=0$};\\
1,&\text{if $t<0$},
\end{cases}
\end{equation*}
where the implied constants depend only on $n,s,t$.
\end{lemma}

\section{Reproducing Kernels and Fractional Differential Operators}
\label{SDiffOp}

In this section we review the properties of the reproducing
kernels and define a family of differential operators $D^t_s$.

Denote by $H_m(\mathbb R^n)$ the space of all homogeneous
(Euclidean) harmonic polynomials of degree $m$ on $\mathbb R^n$.
It is finite dimensional with $\text{dim}\, H_m\sim m^{n-2}$ $(m\geq 1)$.
By homogeneity, $q_m\in H_m(\mathbb R^n)$ is determined by its
restriction on $\mathbb S$ which is called a spherical harmonic
and the space of spherical harmonics of degree $m$ is denoted
by $H_m(\mathbb S)$.
Spherical harmonics of different degrees are orthogonal
on $L^2(\mathbb S)$,
\begin{equation}\label{orthog}
\int_{\mathbb S} q_m(\zeta)q_k(\zeta)\,d\sigma(\zeta)=0
\qquad (m\neq k, q_m\in H_m(\mathbb S), q_k\in H_k(\mathbb S)).
\end{equation}
Point evaluation functionals are bounded on $H_m(\mathbb S)$ and
so for every $\eta\in\mathbb S$, there exists
$Z_m(\eta,\cdot)\in H_m(\mathbb S)$, called the zonal harmonic
of degree $m$ with pole $\eta$, such that for all
$q_m\in H_m(\mathbb S)$,
\begin{equation}\label{zonal}
q_m(\eta)=\int_{\mathbb S}q_m(\zeta)Z_m(\eta,\zeta)\,d\sigma(\zeta).
\end{equation}
$Z_m(\cdot,\cdot)$ is real valued, symmetric, and homogeneous of degree
$m$ in each variable.
On the diagonal, $Z_m(\zeta,\zeta)=\text{dim}\,H_m$, and in general
$\lvert Z_m(\eta,\zeta)\rvert\leq Z_m(\zeta,\zeta)$.
Thus
\begin{equation}\label{Zmless}
\lvert Z_m(\eta,\zeta)\rvert\lesssim m^{n-2}
\qquad(m\geq 1).
\end{equation}
For details we refer the reader to \cite[Chapter 5]{ABR}.

For $q_m\in H_m(\mathbb R^n)$, the solution of the $\mathcal H$-harmonic
Dirichlet problem on $\mathbb B$ with boundary data
$q_m\vert_{\mathbb S}$ is given by (\cite[Theorem 6.1.1]{St1})
\begin{equation}\label{solnd}
g(x)=S_m(\lvert x\rvert)q_m(x)
\qquad (x\in\overline{\mathbb B}).
\end{equation}
That is, $g$ is $\mathcal H$-harmonic on $\mathbb B$, continuous on
$\overline{\mathbb B}$ and equals $q_m$ on $\mathbb S$.
Here, the factor $S_m(r)$ $(0\leq r\leq 1)$ is given by
\begin{equation}\label{Smr}
S_m(r)
=\frac{F(m,1-\tfrac{1}{2}n;m+\tfrac{1}{2}n;r^2)}
{F(m,1-\tfrac{1}{2}n;m+\tfrac{1}{2}n;1)},
\end{equation}
where
\begin{equation}\label{hyperg}
F(a,b;c;z)=\sum_{k=0}^\infty\frac{(a)_k(b)_k}{(c)_k k!}\,z^k
\end{equation}
is the Gauss hypergeometric function.
$S_m$ depends also on the dimension $n$ but we do not write this
for shortness.
When the dimension $n$ is even the hypergeometric series terminates
and $S_m$ is a polynomial, but this is not true in odd dimensions.
$S_m(r)$ is a decreasing function of $r$, and is normalized so that
$S_m(1)=1$.
$S_0\equiv 1$ and when $m\geq 1$, the estimate
(\cite[Proposition I.6]{Sou} or \cite[Lemma 2.6]{St3})
\begin{equation}\label{Smbound}
1\leq S_m(r)\leq Cm^{n/2-1}
\qquad (0\leq r\leq 1)
\end{equation}
holds, where $C=C(n)$ is a constant depending only on $n$.

Every $\mathcal H$-harmonic function on $\mathbb B$ can be written
as a series of terms of the form \eqref{solnd}.
More precisely, for every $f\in\mathcal H(\mathbb B)$, there exists
a unique sequence of polynomials $q_m\in H_m(\mathbb R^n)$ such that
(see \cite{Ja1}, \cite{JP2}, \cite{Min}, \cite[Theorem 6.3.1]{St1})
\begin{equation*}
f(x)=\sum_{m=0}^\infty S_m(\lvert x\rvert)q_m(x)
\qquad (x\in\mathbb B),
\end{equation*}
where the series converges absolutely and uniformly on compact
subsets of $\mathbb B$.

The hyperbolic Poisson kernel and its series expansion are
given by \cite[Theorem 6.2.2]{St1}
\begin{equation}\label{Poisson}
P_h(x,\zeta)
=\frac{(1-\lvert x\rvert^2)^{n-1}}{\lvert x-\zeta\rvert^{2(n-1)}}
=\sum_{m=0}^\infty S_m(\lvert x\rvert)Z_m(x,\zeta)
\qquad (x\in\mathbb B,\zeta\in\mathbb S).
\end{equation}
For the Bergman reproducing kernels $\mathcal R_\alpha(x,y)$
a closed formula is not known, however, the following series
expansion holds (\cite[Corollary III.5]{Sou},
\cite[Theorem 5.3]{St2})
\begin{equation}\label{KernelExp}
\mathcal R_\alpha(x,y)
=\sum_{m=0}^\infty c_m(\alpha)S_m(\lvert x\rvert)
S_m(\lvert y\rvert)Z_m(x,y)
\qquad (\alpha>-1, x,y\in\mathbb B),
\end{equation}
where the coefficients $c_m(\alpha)$ are determined by
\begin{equation}\label{cm}
\frac{1}{c_m(\alpha)}
=n\int_0^1 r^{2m+n-1}S_m^2(r)(1-r^2)^\alpha\,dr.
\end{equation}
A formula for the above integral is not known either.
However, the estimate
\begin{equation}\label{cmasym}
c_m(\alpha)\sim m^{\alpha+1}\qquad (m\to\infty)
\end{equation}
holds (see (\cite[Theorem III.6]{Sou}) from which it follows
that the series in \eqref{KernelExp} converges absolutely and
uniformly on $K\times\overline{\mathbb B}$ for every compact
$K\subset\mathbb B$.

It is clear from \eqref{cm} that $c_m(\alpha)>0$.
Using these coefficients we define a family of fractional
differential operators following \cite{JP1} and \cite{GKU}.

\begin{definition}
Let $s>-1$ and $s+t>-1$.
If $f\in\mathcal H(\mathbb B)$ has the series expansion
$f(x)=\sum_{m=0}^\infty S_m(\lvert x\rvert)q_m(x)$,
then define
\begin{equation}\label{DefineDst}
D^t_s f(x)
=\sum_{m=0}^\infty\frac{c_m(s+t)}{c_m(s)}S_m(\lvert x\rvert)q_m(x).
\end{equation}
\end{definition}

The operator $D^t_s$ multiplies the $m^{\text{th}}$ term of the
series expansion of $f$ with the coefficient
$c_m(s+t)/c_m(s)\sim m^t$ by \eqref{cmasym}.
Similar types of operators are frequently used in the theory of
holomorphic and harmonic Bergman spaces and act as
differential operators of order $t$ (integral if $t<0$).
For $\mathcal H$-harmonic functions slightly different operators with
multipliers $\Gamma(m+s)/\Gamma(m+\beta+s)$ are used in
\cite{RKSL} and Hardy-Littlewood inequalities are obtained.

\begin{lemma}\label{LDstcont}
For $f\in\mathcal H(\mathbb B)$, the series in \eqref{DefineDst}
absolutely and uniformly converges on compact subsets of
$\mathbb B$ and so $D^t_sf\in\mathcal H(\mathbb B)$.
In addition,
$D^t_s:\mathcal H(\mathbb B)\to\mathcal H(\mathbb B)$ is
continuous when $\mathcal H(\mathbb B)$ is equipped with the
topology of uniform convergence on compact subsets.
\end{lemma}

This lemma can be verified in the same way as
\cite[Theorems 3.1 and 3.2]{GKU}.
An additional term $S_m(r)$ appears but it can easily be
handled with the estimate \eqref{Smbound}.

The operator $D^t_s$ is invertible with
\begin{equation}\label{DstInverse}
D^{-t}_{s+t}D^t_s=D^t_sD^{-t}_{s+t}=\textrm{Id}.
\end{equation}
The role of $s$ is minor and one reason for its inclusion is to
simplify the action of $D^t_s$ on the reproducing kernel $\mathcal R_s$,
\begin{equation}\label{DstRs}
D^t_s\mathcal R_s(x,y)=\mathcal R_{s+t}(x,y).
\end{equation}

If $f\in\mathcal H(\mathbb B)$ is also integrable, then
$D^t_s$ can be written as an integral operator.

\begin{lemma}\label{LDstInt}
Let $s>-1$, $s+t>-1$ and $f\in L^1_s(\mathbb B)$.
\begin{enumerate}
\item[(i)] $D^t_sP_sf(x)
=D^t_s\int_{\mathbb B} \mathcal R_s(x,y)f(y)d\nu_s(y)
=\int_{\mathbb B}\mathcal R_{s+t}(x,y)f(y)d\nu_s(y).$
\item[(ii)] If $f$ is also in $\mathcal H(\mathbb B)$, then
$D^t_sf(x)=\int_{\mathbb B} \mathcal R_{s+t}(x,y)f(y)d\nu_s(y).$
\end{enumerate}
\end{lemma}

\begin{proof}
For fixed $x\in\mathbb B$, the series in \eqref{KernelExp} uniformly
converges for $y\in\mathbb B$.
Thus
\begin{align}\label{SeriesofPs}
\int_{\mathbb B} \mathcal R_s(x,y)f(y)\,d\nu_s(y)
&=\sum_{m=0}^\infty c_m(s)S_m(\lvert x\rvert)
\int_{\mathbb B}Z_m(x,y)S_m(\lvert y\rvert)f(y)\,d\nu_s(y)\notag\\
&=:\sum_{m=0}^\infty c_m(s)S_m(\lvert x\rvert)q_m(x).
\end{align}
The function $q_m$ is in $H_m(\mathbb R^n)$ and the series
in \eqref{SeriesofPs} absolutely and uniformly converges on
compact subsets of $\mathbb B$.
This follows from \eqref{Smbound} and \eqref{cmasym}, and
the fact that
$\lvert Z_m(x,y)\rvert\lesssim \lvert x\rvert^m m^{n-2}$ by \eqref{Zmless}.
Thus, the series in \eqref{SeriesofPs} is the (unique) series expansion
and by \eqref{DefineDst},
$D^t_s\int_{\mathbb B} \mathcal R_s(x,y)f(y)d\nu_s(y)
=\sum_{m=0}^\infty c_m(s+t)S_m(\lvert x\rvert)q_m(x)$.
This series equals
$\int_{\mathbb B}\mathcal R_{s+t}(x,y)f(y)d\nu_s(y)$
by the same reasoning.

If $f\in L^1_s\cap\mathcal H(\mathbb B)$, then $P_sf=f$ because
the reproducing property in \eqref{Reproduce} holds also for
$f\in \mathcal B^1_\alpha$ (\cite[Lemma 7.1]{U1}).
\end{proof}

The following upper estimates of the reproducing kernels
$\mathcal R_\alpha$ have been obtained in
\cite[Theorem 1.2]{U1}.
Here, $\nabla_x$ means that the gradient is taken with respect to $x$.
\begin{lemma}\label{LEstR}
Let $\alpha>-1$.
There exists a constant $C=C(n,\alpha)>0$ such that
for all $x,y\in\mathbb B$,
\begin{enumerate}
  \item[(a)] $\lvert\mathcal R_\alpha(x,y)\rvert
  \leq\dfrac{C}{[x,y]^{\alpha+n}}$,
  \item[(b)] $\lvert\nabla_x \mathcal R_\alpha(x,y)\rvert
  \leq\dfrac{C}{[x,y]^{\alpha+n+1}}$.
\end{enumerate}
\end{lemma}

These estimates lead to the following projection theorem
(see \cite[Theorem 1.1]{U1}.

\begin{lemma}\label{Lproject}
Let $1\leq p<\infty$ and $\alpha,\gamma>-1$.
The operator $P_\gamma\colon L^p_\alpha\to\mathcal B^p_\alpha$
is bounded if and only if $\alpha+1<p(\gamma+1)$.
In this case $P_\gamma f=f$ for
$f\in L^p_\alpha\cap\mathcal H(\mathbb B)$.
\end{lemma}

\section{Elementary Properties of the Bloch Space}

We first mention a few basic basic properties of $\mathcal B$ and
$\mathcal B_0$.
The verifications are omitted as they are straightforward and are
similar to the holomorphic or the Euclidean harmonic case.
The space $\mathcal B$ is a Banach space with respect to the norm
$\|\cdot\|_{\mathcal B}$ and $\mathcal B_0$ is a closed subspace of
$\mathcal B$.
The seminorm $p_{\mathcal B}$ is M\"obius invariant, i.e.,
$p_{\mathcal B}(f\circ\psi)=p_{\mathcal B}(f)$ for all
$\psi\in\mathcal M(\mathbb B)$.
If $f\in\mathcal B$, then for every $x,y\in\mathbb B$,
\begin{equation}\label{BlochDif}
\lvert f(x)-f(y)\rvert
\leq \frac{1}{2}\,p_{\mathcal B}(f)\beta(x,y).
\end{equation}
In particular, taking $y=0$ and using
$\beta(x,0)\leq 1+\log 1/(1-\lvert x\rvert)$ shows
\begin{equation}\label{Blochpoint}
\lvert f(x)\rvert
\leq\|f\|_{\mathcal B}\Big(1+\log\frac{1}{1-\lvert x\rvert}\Big)
\qquad (f\in\mathcal B, x\in \mathbb B).
\end{equation}

There are various results in \cite{GJ} and \cite{Ja2} that show
that $\mathcal H$-harmonic functions can have different behaviours
depending on whether the dimension $n$ is odd or even.
We show here one more difference.
Let $q_m\in H_m(\mathbb S)$.
If the dimension $n$ is even, $S_m(\lvert x\rvert)$ is a
polynomial and the Poisson extension
$P_h[q_m](x)=S_m(\lvert x\rvert)q_m(x)$ is
an $\mathcal H$-harmonic polynomial.
This is not true when the dimension $n$ is odd.
In fact, in this case a non-constant polynomial can not be
$\mathcal H$-harmonic on $\mathbb B$.

\begin{lemma}\label{Lpolyhar}
In odd dimensions, there are no non-constant polynomials in
$\mathcal H(\mathbb B)$.
\end{lemma}

\begin{proof}
Suppose that $p$ is a polynomial of degree $M\geq1$ and
$p\in\mathcal H(\mathbb B)$.
Then $p=\sum_{m=0}^M p^{(m)}$, where $p^{(m)}$ is a homogeneous
polynomial of degree $m$ and
by \cite[Theorem 5.7]{ABR}, each $p^{(m)}$ can be written in the form
\begin{equation*}
p^{(m)}(x)=\sum_{j=0}^{[m/2]}\lvert x\rvert^{2j}q_{m-2j}^{(m)}(x),
\end{equation*}
with $q_{m-2j}^{(m)}\in H_{m-2j}(\mathbb R^n)$.
Thus
\begin{equation}\label{poly1}
p(x)=\sum_{m=0}^M\sum_{j=0}^{[m/2]}\lvert x\rvert^{2j}q_{m-2j}^{(m)}(x)
=\sum_{k=0}^M\sum_{j=0}^{[(M-k)/2]}\lvert x\rvert^{2j}q_{k}^{(k+2j)}(x),
\end{equation}
where we make the change of index $k=m-2j$.
On the boundary $\mathbb S$,
\begin{equation*}
p(\zeta)=\sum_{k=0}^M\sum_{j=0}^{[(M-k)/2]}q_{k}^{(k+2j)}(\zeta)
=:\sum_{k=0}^M Q_k(\zeta)
\qquad (\zeta\in\mathbb S),
\end{equation*}
with $Q_k\in H_k(\mathbb S)$, and since $p$ is $\mathcal H$-harmonic,
$p=P_h[p\vert_{\mathbb S}]$ and so
\begin{equation}\label{poly2}
p(x)=\sum_{k=0}^M S_k(\lvert x\rvert)Q_k(x)
\qquad (x\in\mathbb B).
\end{equation}
Because $p$ is non-constant, there exist $1\leq i\leq M$ and
$\eta\in\mathbb S$ such that $Q_i(\eta)\neq 0$.
For $0\leq r\leq 1$, we compute the integral
\begin{equation*}
I=\int_{\mathbb S}p(r\zeta)Z_i(\eta,\zeta)\,d\sigma(\zeta)
\end{equation*}
in two ways, using \eqref{poly1} and then \eqref{poly2}.
First, by \eqref{poly1},
\begin{equation*}
I=\sum_{k=0}^M\sum_{j=0}^{[(M-k)/2]}r^{k+2j}
\int_{\mathbb S}q_{k}^{(k+2j)}(\zeta)Z_i(\eta,\zeta)\,d\sigma(\zeta)
=r^{i}\sum_{j=0}^{[(M-i)/2]} r^{2j}q_i^{(i+2j)}(\eta),
\end{equation*}
by \eqref{orthog} and \eqref{zonal}.
Next, by \eqref{poly2},
\begin{equation*}
I=\sum_{k=0}^M S_k(r) r^k
\int_{\mathbb S}Q_k(\zeta)Z_i(\eta,\zeta)\,d\sigma(\zeta)
=r^{i}S_i(r)Q_i(\eta).
\end{equation*}
Combining these we see that
$S_i(r)Q_i(\eta)=\sum_{j=0}^{[(M-i)/2]} r^{2j}q_i^{(i+2j)}(\eta)$
for all $0<r\leq 1$.
Because $Q_i(\eta)\neq 0$, this shows that $S_i(r)$ is a
polynomial of $r$.
This is a contradiction because when the dimension $n$ is odd,
$S_i$ is not a polynomial for $i\geq 1$ since its
hypergeometric series do not terminate.
\end{proof}

Let $q_m\in H_m(\mathbb R^n)$.
When the dimension $n$ is even,
$P_h[q_m\vert_{\mathbb S}](x)=S_m(\lvert x\rvert)q_m(x)$ is
a polynomial and so is in $\mathcal B_0$.
It is also in $\mathcal B_0$ when the dimension is odd.
This follows from the following two elementary facts about
hypergeometric series.
First,
\begin{equation}\label{DerivHyp}
\frac{d}{dz}F(a,b;c;z)=\frac{ab}{c}F(a+1,b+1;c+1;z)
\end{equation}
and second, if $\Re\{c-a-b\}>0$, then $F(a,b;c;z)$ uniformly converges
and so is bounded on the closed disk $\{z:\lvert z\rvert\leq 1\}$.

\begin{lemma}\label{LSmpminB0}
If $q_m\in H_m(\mathbb R^n)$, then
$P_h[q_m\vert_{\mathbb S}](x)=S_m(\lvert x\rvert)q_m(x)$ is in $\mathcal B_0$.
\end{lemma}

\begin{proof}
The case $n=2$ is obvious and we assume $n\geq 3$.
By \eqref{DerivHyp},
\begin{align*}
\frac{\partial}{\partial x_i}S_m(\lvert x\rvert)q_m(x)
={}&2x_i\frac{m(1-\tfrac{1}{2}n)}{m+\tfrac{1}{2}n}\,
\frac{F(m+1,2-\tfrac{1}{2}n;m+\tfrac{1}{2}n+1;\lvert x\rvert^2)}
{F(m,1-\tfrac{1}{2}n;m+\tfrac{1}{2}n;1)}\,q_m(x)\\
&+S_m(\lvert x\rvert)\frac{\partial}{\partial x_i}q_m(x).
\end{align*}
The hypergeometric function in the first term is bounded since
$\Re\{c-a-b\}=n-2>0$.
Since the second term is also bounded, the result follows.
\end{proof}

We next show that Bergman projections of polynomials are in $\mathcal B_0$.

\begin{lemma}\label{LBergmanpoly}
$P_\alpha p\in\mathcal B_0$ for every polynomial $p$ and $\alpha>-1$.
\end{lemma}

\begin{proof}
We can assume $p$ is homogeneous.
By \cite[Theorem 5.7]{ABR} again, $p$ can be written in the form
$p=q_m+\lvert x\rvert^2 q_{m-2}+\cdots+\lvert x\rvert^{2k} q_{m-2k}$,
where $k=[m/2]$ and $q_j\in H_j(\mathbb R^n)$.
Thus, it suffices to show that $P_\alpha(\lvert x\rvert^k q_j)\in\mathcal B_0$
for every $k\geq 0$ and $q_j\in H_j(\mathbb R^n)$.
Now, using the uniform convergence of the series  in \eqref{KernelExp}, integrating
in polar coordinates and then using \eqref{orthog} and \eqref{zonal}, we obtain
\begin{align*}
P_{\alpha}&\bigl(\lvert x\rvert^k q_j\bigr)(x)
=\int_\mathbb B\mathcal R_\alpha(x,y)\lvert y\rvert^k q_j(y)\,d\nu_\alpha(y)\\
&=\sum_{m=0}^\infty c_m(\alpha)S_m(\lvert x\rvert)
\int_0^1 nr^{n-1}S_m(r)r^{k+m+j}(1-r^2)^\alpha
\int_{\mathbb S} Z_m(x,\zeta)q_j(\zeta)\,d\sigma(\zeta)dr\\
&=c_j(\alpha)S_j(\lvert x\rvert)q_j(x)\int_0^1 nr^{n-1}S_j(r)r^{k+2j}(1-r^2)^\alpha\,dr\\
&=CS_j(\lvert x\rvert)q_j(x),
\end{align*}
which belongs to $\mathcal B_0$ by Lemma \ref{LSmpminB0}.
\end{proof}

\begin{remark}\label{Rpolyproj}
The above proof shows also that for every polynomial $p$ and $\alpha>-1$,
$P_\alpha p\in\text{span}\bigl\{S_m(\lvert x\rvert)q_m(x)
\,\vert\, q_m\in H_m(\mathbb R^n), m=0,1,2\dotsc \bigr\}$.
\end{remark}

\section{Projections onto the Bloch and the Little Bloch Space}

The aim of this is section is to prove Theorem \ref{Tproj}.
The main issue is to show that the projection
$P_\alpha\colon L^\infty\to\mathcal B$ is onto.
To achieve this we first characterize $\mathcal B$ and $\mathcal B_0$
in terms of the differential operators $D^t_s$.

\begin{lemma}\label{LPintoB}
For every $\alpha>-1$, $P_\alpha\colon L^\infty(\mathbb B)\to\mathcal B$
is bounded.
In addition, if $f\in C(\overline{\mathbb B})$,
then $P_\alpha f\in\mathcal B_0$.
\end{lemma}

\begin{proof}
The estimate in Lemma \ref{LEstR}(b) together with Lemma \ref{LIntonSandB}
immediately implies that $P_\alpha$ maps $L^\infty(\mathbb B)$ boundedly
into $\mathcal B$ (see \cite[Theorem 1.5]{U1}).
That $P_\alpha$ maps $C(\overline{\mathbb B})$ into $\mathcal B_0$ follows
from the Stone-Weierstrass theorem and Lemma \ref{LBergmanpoly}.
\end{proof}

We next obtain two estimates.
One is similar to Lemma \ref{LIntonSandB}, and
the other to \cite[Lemma 4.3]{U2}, but include an extra term
$\beta(x,y)$, the hyperbolic distance between $x$ and $y$.

\begin{lemma}\label{Lest1}
Let $s>-1$ and $t>0$.
\begin{enumerate}
\item[(i)] There exists a constant $C=C(n,s,t)>0$
such that for all $x\in\mathbb B$,
\begin{equation*}
\int_{\mathbb B}\frac{\beta(x,y)\,(1-\lvert y\rvert^2)^s}{[x,y]^{n+s+t}}\,
d\nu(y)\leq\frac{C}{(1-\lvert x\rvert^2)^t}.
\end{equation*}
\item[(ii)] Given $\varepsilon>0$, there exists $0<r_\varepsilon<1$ such that
for all $r$ with $r_\varepsilon<r<1$ and all $x\in\mathbb B$,
\begin{equation*}
\int_{\mathbb B\backslash E_r(x)}
\frac{\beta(x,y)\,(1-\lvert y\rvert^2)^s}{[x,y]^{n+s+t}}\,d\nu(y)
<\frac{\varepsilon}{(1-\lvert x\rvert^2)^t}.
\end{equation*}
\end{enumerate}
\end{lemma}

\begin{proof}
The proof is similar to the proof of \cite[Lemma 4.3]{U2}, requiring
only a minor modification.
For $0\leq r<1$, let
\begin{equation*}
I_r(x)=(1-\lvert x\rvert^2)^t
\int_{\mathbb B\backslash E_r(x)}
\frac{\beta(x,y)\,(1-\lvert y\rvert^2)^s}{[x,y]^{n+s+t}}\,d\nu(y),
\end{equation*}
where for $r=0$, $\mathbb B\backslash E_0(x)=\mathbb B$.
In the integral make the change of variable $y=\varphi_x(z)$.
Note that
$\varphi_x^{-1}(\mathbb B\backslash E_r(x))=\mathbb B\backslash\mathbb B_r$.
Employing \eqref{MobiusIdnt}, \eqref{Jacob} and \eqref{avarphia},
and the fact that $\beta(x,\varphi_x(z))=\beta(0,z)$
by the M\"obius-invariance of $\beta$, we obtain that
\begin{equation*}
I_r(x)=
\int_{\mathbb B\backslash \mathbb B_r}
\frac{\beta(0,z)\,(1-\lvert z\rvert^2)^s}{[x,z]^{n+s-t}}\,d\nu(z).
\end{equation*}
Next, using $\beta(0,z)\leq 1+\log1/(1-\lvert z\rvert)$ and
integrating in polar coordinates shows
\begin{equation*}
I_r(x)\leq\int_r^1 n\tau^{n-1}\Big(1+\log\frac{1}{1-\tau}\Big)\,(1-\tau^2)^s
\int_{\mathbb S}\frac{d\sigma(\zeta)}{\lvert \tau x-\zeta\rvert^{n+s-t}}\,d\tau.
\end{equation*}
Estimating the inner integral with Lemma \ref{LIntonSandB} in three cases
and using the inequality $1-\tau^2\lvert x\rvert^2\geq 1-\tau^2$, we see that
\begin{equation*}
\int_{\mathbb S} \frac{d\sigma(\zeta)}{\lvert \tau x-\zeta\rvert^{n+s-t}}
\leq C g(\tau):=
\begin{cases}
\dfrac{1}{(1-\tau^2)^{1+s-t}},&\text{if $1+s-t>0$};\\
1+\log\dfrac{1}{1-\tau^2},&\text{if $1+s-t=0$};\\
1,&\text{if $1+s-t<0$},
\end{cases}
\end{equation*}
where $C$ depends only on $n,s,t$.
Thus
\begin{equation*}
I_r(x)
\leq C\int_r^1 n\tau^{n-1}\Big(1+\log\frac{1}{1-\tau}\Big)
\,(1-\tau^2)^s g(\tau)\,d\tau.
\end{equation*}
Because $s>-1$ and $t>0$, in all the three cases, the above
integral is finite when $r=0$.
This proves both parts of the lemma.
\end{proof}

We prove one more estimate.
In the lemma below we consider the integral over $r\mathbb B$
for $0<r\leq 1$, not just $\mathbb B$.
This generality will be needed later.

\begin{lemma}\label{LEstrB}
Let $s>-1$ and $t>0$.
There exists a constant $C=C(n,s,t)>0$ such that for all $0<r\leq 1$
and $f\in\mathcal B$,
\begin{equation*}
(1-\lvert x\rvert^2)^t\,\Bigl\lvert\int_{r\mathbb B}
\mathcal R_{s+t}(x,y) f(y)\,d\nu_{s}(y)\Bigr\rvert
\leq C\|f\|_{\mathcal B}.
\end{equation*}
\end{lemma}

\begin{proof}
We write
\begin{align*}
\int_{r\mathbb B} \mathcal R_{s+t}(x,y)f(y)d\nu_s(y)
&=\int_{r\mathbb B} \mathcal R_{s+t}(x,y)f(x)d\nu_s(y)\\
&+\int_{r\mathbb B} \mathcal R_{s+t}(x,y)\big(f(y)-f(x)\big)d\nu_s(y)
=:h_{1,r}(x)+h_{2,r}(x).
\end{align*}
Integrating in polar coordinates gives
\begin{equation*}
h_{1,r}(x)=f(x)\int_0^r n\tau^{n-1}(1-\tau^2)^s
\int_{\mathbb S}\mathcal R_{s+t}(x,\tau\zeta)\,d\sigma(\zeta)\,d\tau.
\end{equation*}
Now, by the mean-value theorem for $\mathcal H$-harmonic functions,
the inner integral is $\mathcal R_{s+t}(x,0)$ and by
\eqref{KernelExp}, $\mathcal R_{s+t}(x,0)=c_0(s+t)$ for all
$x\in\mathbb B$ because $Z_m(x,0)=0$ for $m\geq 1$, $Z_0\equiv 1$
and $S_0\equiv 1$.
Therefore, using also \eqref{Blochpoint}, we obtain
\begin{align}\label{Esth1}
\begin{split}
\bigl\lvert h_{1,r}(x)\bigr\rvert
&\leq c_0(s+t)\lvert f(x)\rvert
\int_0^1n\tau^{n-1}(1-\tau^2)^s\,d\tau
=C\lvert f(x)\rvert\\
&\leq C\|f\|_{\mathcal B}
\Bigl(1+\log\frac{1}{1-\lvert x\rvert}\Bigr),
\end{split}
\end{align}
with $C$ depending only on $n,s,t$.
This shows
$(1-\lvert x\rvert^2)^t\lvert h_{1,r}(x)\rvert
\lesssim \|f|_{\mathcal B}$.

Next, by \eqref{BlochDif} and Lemma \ref{LEstR}(a),
\begin{equation*}
\lvert h_{2,r}(x)\rvert
\lesssim p_{\mathcal B}(f)\int_{r\mathbb B}
\frac{\beta(x,y)\,d\nu_s(y)}{[x,y]^{n+s+t}}
\lesssim \|f\|_{\mathcal B}\int_{\mathbb B}
\frac{\beta(x,y)\,d\nu_s(y)}{[x,y]^{n+s+t}}
\lesssim \frac{\|f\|_{\mathcal B}}{(1-\lvert x\rvert^2)^t},
\end{equation*}
where in the last inequality we use Lemma \ref{Lest1}(i).
This proves the lemma.
\end{proof}

\begin{proposition}\label{PDstchar}
Let $s>-1$, $t>0$ and $f\in\mathcal H(\mathbb B)$.
\begin{enumerate}
\item[(i)] $f\in\mathcal B$ if and only if
$(1-\lvert x\rvert^2)^tD^t_sf(x)\in L^\infty(\mathbb B)$.
In this case,
\begin{equation*}
  \|f\|_{\mathcal B}\sim\|(1-\lvert x\rvert^2)^tD^t_sf(x)\|_{L^\infty},
\end{equation*}
where the implied constants depend only on $n,s,t$, and are
independent of $f$.
In addition,
\begin{equation}\label{PsIst}
P_s\bigl[(1-\lvert x\rvert^2)^tD^t_sf(x)\bigr]=f.
\end{equation}
\item[(ii)] $f\in\mathcal B_0$ if and only if
$(1-\lvert x\rvert^2)^tD^t_sf(x)\in C_0(\mathbb B)$.
\end{enumerate}
\end{proposition}

This proposition verifies the onto part of Theorem \ref{Tproj}
and completes its proof.
It also shows that for every $s>-1$ and $t>0$,
$\|(1-\lvert x\rvert^2)^tD^t_sf(x)\|_{L^\infty}$ is a norm on
$\mathcal B$ equivalent to $\|f\|_{\mathcal B}$.
In the rest of the paper we mostly employ these norms as they are
easier to work with than $\|\cdot\|_{\mathcal B}$.

\begin{proof}
Suppose $f\in\mathcal B$.
Then $f\in L^1_s\cap\mathcal H(\mathbb B)$ by \eqref{Blochpoint}
and by Lemma \ref{LDstInt}(ii) we have
$D^t_s f(x)=\int_{\mathbb B}\mathcal R_{s+t}(x,y)f(y)d\nu_s(y)$.
That
$\|(1-\lvert x\rvert^2)^t D^t_sf(x)\|_{L^\infty}\lesssim \|f\|_{\mathcal B}$
follows now from Lemma \ref{LEstrB}.

Suppose now that $(1-\lvert x\rvert^2)^tD^t_sf(x)\in L^\infty(\mathbb B)$.
We first show \eqref{PsIst} which holds because
$D^t_sf\in L^1_{s+t}(\mathbb B)\cap\mathcal H(\mathbb B)$
and by Lemma \ref{LDstInt}(ii) and \eqref{DstInverse},
\begin{equation*}
P_s\bigl[(1-\lvert x\rvert^2)^tD^t_sf(x)\bigr](x)
=\int_{\mathbb B} \mathcal R_s(x,y)D^t_sf(y)\,d\nu_{s+t}(y)
=D^{-t}_{s+t}(D^t_sf)(x)=f(x).
\end{equation*}
It now follows from Lemma \ref{LPintoB} that $f\in\mathcal B$
and $\|f\|_{\mathcal B}
\leq \|P_s\|\,\|(1-\lvert x\rvert^2)^tD^t_sf(x)\|_{L^\infty}$.

(ii) Suppose $f\in\mathcal B_0$.
For $\varepsilon>0$, pick $r>r_\varepsilon$ with $r_\varepsilon$
is as given in Lemma  \ref{Lest1} (ii).
As in the proof of Lemma \ref{LEstrB}, we write
$D^t_s f(x)=\int_{\mathbb B}\mathcal R_{s+t}(x,y)f(y)d\nu_s(y)$
in the form
\begin{align*}
D^t_sf(x)
={}&\int_{\mathbb B} \mathcal R_{s+t}(x,y)f(x)\,d\nu_s(y)
+\int_{\mathbb B\backslash E_r(x)}
\mathcal R_{s+t}(x,y)\big(f(y)-f(x)\big)\,d\nu_s(y)\\
&+\int_{E_r(x)} \mathcal R_{s+t}(x,y)\big(f(y)-f(x)\big)\,d\nu_s(y)
=:h_1(x)+h_2(x)+h_3(x).
\end{align*}
We have $(1-\lvert x\rvert^2)^th_1(x)\in C_0(\mathbb B)$ by \eqref{Esth1}.
Next, applying \eqref{BlochDif}, Lemma \ref{LEstR}(a) and then
Lemma \ref{Lest1}(ii) shows that for some constant $C=C(n,s,t)$,
\begin{equation*}
\lvert h_2(x)\rvert
\leq C p_{\mathcal B}(f)\int_{\mathbb B\backslash E_r(x)}
\frac{\beta(x,y)\,(1-\lvert y\rvert^2)^s}{[x,y]^{n+s+t}}\,d\nu(y)
<C p_{\mathcal B}(f)\frac{\varepsilon}{(1-\lvert x\rvert^2)^t}.
\end{equation*}
Thus $(1-\lvert x\rvert^2)^t\lvert h_2(x)\rvert\lesssim \varepsilon$.

To estimate $h_3$ note that for $y\in E_r(x)$, by the mean-value inequality,
\begin{equation*}
\lvert f(y)-f(x)\rvert\leq \lvert y-x\rvert\sup_{z\in E_r(x)}\lvert\nabla f(z)\rvert
\end{equation*}
and by \eqref{phmetric},
$\lvert y-x\rvert=\rho(x,y)[x,y]<r[x,y]\lesssim r(1-\lvert x\rvert^2)$ since
by part (ii) of Lemma \ref{Lratiobracket}, $[x,y]\sim[x,x]=1-\lvert x\rvert^2$.
Therefore
\begin{equation*}
\lvert f(y)-f(x)\rvert
\lesssim (1-\lvert x\rvert^2)\sup_{z\in E_r(x)}\lvert\nabla f(z)\rvert
\lesssim \sup_{z\in E_r(x)}(1-\lvert z\rvert^2)\lvert\nabla f(z)\rvert,
\end{equation*}
where the last inequality follows from Lemma \ref{Lratiobracket}(i).
Hence, by Lemma \ref{LIntonSandB},
\begin{align*}
(1-\lvert x\rvert^2)^t\lvert h_3(x)\rvert
&\lesssim \sup_{z\in E_r(x)}(1-\lvert z\rvert^2)\lvert\nabla f(z)\rvert\,
(1-\lvert x\rvert^2)^t\int_{\mathbb B}\frac{d\nu_s(y)}{[x,y]^{n+s+t}}\\
&\lesssim \sup_{z\in E_r(x)}(1-\lvert z\rvert^2)\lvert\nabla f(z)\rvert.
\end{align*}
Now, by \eqref{Ercenter}, for $z\in E_r(x)$ we have
$\lvert z\rvert
\geq \dfrac{\bigl\lvert\lvert x\rvert-r\bigr\rvert}{1-r\lvert x\rvert}$
and the right-hand side tends to $1$ as $\lvert x\rvert\to 1^-$.
Since $f\in\mathcal B_0$, this shows that
$\lim_{\lvert x\rvert\to 1^-}(1-\lvert x\rvert^2)^t\lvert h_3(x)\rvert=0$;
and we conclude that $(1-\lvert x\rvert^2)^tD^t_sf(x)\in C_0(\mathbb B)$.

Conversely, if $(1-\lvert x\rvert^2)^tD^t_sf(x)\in C_0(\mathbb B)$, then
$f\in\mathcal B_0$ by Lemma \ref{LPintoB} and \eqref{PsIst}.
\end{proof}

The next corollary is the $\mathcal H$-harmonic counterpart of the fact that
harmonic polynomials are dense in the harmonic little Bloch space.

\begin{corollary}\label{Cdense}
$\text{span}\bigl\{S_m(\lvert x\rvert)q_m(x)
\,\vert\, q_m\in H_m(\mathbb R^n), m=0,1,2\dotsc \bigr\}$
is dense in $\mathcal B_0$ and therefore $\mathcal B_0$ is separable.
\end{corollary}

\begin{proof}
Pick any $\alpha>-1$ and $t>0$.
If $f\in\mathcal B_0$, then
$\phi(x)=(1-\lvert x\rvert^2)^tD^t_\alpha f(x)$ is in $C_0(\mathbb B)$
and $P_\alpha\phi=f$.
Since $\phi$ can be approximated in $C_0(\mathbb B)$ by polynomials,
the result follows from Remark \ref{Rpolyproj}.
\end{proof}

\section{Duality}

We begin with writing the pairing in \eqref{pairing} as an absolutely
convergent integral.

\begin{lemma}\label{Lpairingabs}
Let $\alpha>-1$, $f\in\mathcal B^1_\alpha$ and $g\in\mathcal B$.
For every $t>0$,
\begin{equation*}
\langle f,g\rangle_\alpha
=\lim_{r\to 1^-}\int_{r\mathbb B}f(x)g(x)\,d\nu_\alpha(x)
=\int_{\mathbb B}f(x)(1-\lvert x\rvert^2)^tD^t_\alpha g(x)\,d\nu_\alpha(x).
\end{equation*}
\end{lemma}

\begin{proof}
Since the reproducing property in \eqref{Reproduce} holds also
in $\mathcal B^1_\alpha$ and
$f\in\mathcal B^1_\alpha\subset\mathcal B^1_{\alpha+t}$,
\begin{equation*}
\lim_{r\to 1^-}\int_{r\mathbb B}f(x)g(x)\,d\nu_\alpha(x)
=\lim_{r\to 1^-}\int_{r\mathbb B}
\int_{\mathbb B}\mathcal R_{\alpha+t}(x,y)f(y)\,d\nu_{\alpha+t}(y)
g(x)\,d\nu_{\alpha}(x),
\end{equation*}
which, after changing the order of the integrals (possible since for
$\lvert x\rvert\leq r$, the functions $\mathcal R_{\alpha+t}(x,y)$
and $g(x)$ are bounded), equals
\begin{equation*}
\lim_{r\to 1^-}\int_{\mathbb B} f(y)(1-\lvert y\rvert^2)^t
\int_{r\mathbb B}\mathcal R_{\alpha+t}(x,y) g(x)\,d\nu_{\alpha}(x)
\,d\nu_\alpha(y).
\end{equation*}
The term $(1-\lvert y\rvert^2)^t
\bigl\rvert\int_{r\mathbb B}\mathcal R_{\alpha+t}(x,y) g(x)
d\nu_{\alpha}(x)\bigr\rvert$
is bounded by a constant independent of $r$ and $y$ by
Lemma \ref{LEstrB} and the fact that $\mathcal R_{\alpha+t}$ is
symmetric.
Thus, by the dominated convergence theorem, we can push the
limit into the integral and obtain
\begin{equation*}
\lim_{r\to 1^-}\int_{r\mathbb B}f(x)g(x)\,d\nu_\alpha(x)
=\int_{\mathbb B} f(y)(1-\lvert y\rvert^2)^t
\int_{\mathbb B}\mathcal R_{\alpha+t}(x,y) g(x)\,d\nu_{\alpha}(x)
\,d\nu_\alpha(y).
\end{equation*}
This gives the desired result since the inner integral is $D^t_\alpha g(y)$
by Lemma \ref{LDstInt}(ii).
\end{proof}

\begin{proof}[Proof of Theorem \ref{Tdual}]
For $g\in\mathcal B$, define
$\Lambda_g\colon\mathcal B^1_\alpha\to\mathbb C$
by $\Lambda_g(f)=\langle f,g\rangle_\alpha$.
Pick some $t>0$.
By Lemma \ref{Lpairingabs} and Proposition \ref{PDstchar},
\begin{equation}\label{pairbound}
\lvert\langle f,g\rangle_\alpha\rvert
\leq \|f\|_{\mathcal B^1_\alpha}
\|(1-\lvert x\rvert^2)^tD^t_\alpha g\|_{L^\infty}
\lesssim \|f\|_{\mathcal B^1_\alpha}\|g\|_{\mathcal B},
\end{equation}
and so $\Lambda_g\in(\mathcal B^1_\alpha)^*$ and
$\|\Lambda_g\|\lesssim \|g\|_{\mathcal B}$.

Conversely, let $\Lambda\in(\mathcal B^1_\alpha)^*$.
Pick $\gamma>\alpha$.
Then, by Lemma \ref{Lproject}, $\Lambda\circ P_\gamma\in (L^1_\alpha)^*$
and by the Riesz representation theorem there exists $\psi\in L^\infty(\mathbb B)$
with $\|\psi\|_{L^\infty}=\|\Lambda\circ P_\gamma\|$ such that for all
$\phi\in L^1_\alpha(\mathbb B)$,
\begin{equation*}
(\Lambda\circ P_\gamma)\phi=\int_{\mathbb B}\phi(y)\psi(y)\,d\nu_\alpha(y).
\end{equation*}
When $\phi=f\in\mathcal B^1_\alpha$, we have $P_\gamma f=f$ and so
\begin{align}
\Lambda(f)
&=\int_{\mathbb B}P_\gamma f(y)\psi(y)\,d\nu_\alpha(y)
=\int_{\mathbb B}\int_{\mathbb B}\mathcal R_\gamma(y,x)f(x)\,d\nu_\gamma(x)
\psi(y)\,d\nu_\alpha(y)\notag\\
&=\int_{\mathbb B}f(x)\int_{\mathbb B}\mathcal R_\gamma(y,x)\psi(y)\,d\nu_\alpha(y)
\,d\nu_\gamma(x)\label{dualeq},
\end{align}
where we can change the order of the integrals because $\psi\in L^\infty$, and
by Lemma \ref{LEstR}(a) and Lemma \ref{LIntonSandB},
$\int_{\mathbb B}\lvert \mathcal R_\gamma(y,x)\rvert\,d\nu_\alpha(y)
\lesssim (1-\lvert x\rvert^2)^{-(\gamma-\alpha)}$.
Let
\begin{equation*}
g(x):=P_\alpha\psi(x)
=\int_{\mathbb B}\mathcal R_\alpha(x,y)\psi(y)\,d\nu_\alpha(y).
\end{equation*}
By Theorem \ref{Tproj}, $g$ is in $\mathcal B$ and
$\|g\|_{\mathcal B}\lesssim \|\psi\|_{L^\infty}=\|\Lambda\circ P_\gamma\|
\lesssim \|\Lambda\|$.
Further, by Lemma \ref{LDstInt}(i),
\begin{equation*}
D^{\gamma-\alpha}_\alpha g(x)
=D^{\gamma-\alpha}_\alpha
\int_{\mathbb B}\mathcal R_\alpha(x,y)\psi(y)\,d\nu_\alpha(y)
=\int_{\mathbb B}\mathcal R_\gamma(x,y)\psi(y)\,d\nu_\alpha(y).
\end{equation*}
Hence by \eqref{dualeq} and the symmetry of
$\mathcal R_\gamma$,
\begin{equation*}
\Lambda(f)=\int_{\mathbb B}f(x)
(1-\lvert x\rvert^2)^{\gamma-\alpha}D^{\gamma-\alpha}_\alpha g(x)
\,d\nu_\alpha(x),
\end{equation*}
which shows that $\Lambda=\Lambda_g$.

To see the uniqueness of $g$, note that for $x_0\in\mathbb B$,
$\mathcal R_\alpha(x_0,\cdot)$ is bounded on $\mathbb B$ and so
belongs to $\mathcal B^1_\alpha$.
In addition, if $g\in\mathcal B$, then $\mathcal R_\alpha(x_0,\cdot)g$
is in $L^1_\alpha(\mathbb B)$ by \eqref{Blochpoint}.
Thus
\begin{align}\label{uniq}
\begin{split}
\Lambda_g(\mathcal R_{\alpha}(x_0,\cdot))
&=\langle \mathcal R_{\alpha}(x_0,\cdot), g\rangle_\alpha
=\lim_{r\to 1^-}\int_{r\mathbb B}
\mathcal R_{\alpha}(x_0,x)g(x)d\nu_\alpha(x)\\
&=\int_{\mathbb B} \mathcal R_{\alpha}(x_0,x)g(x)d\nu_\alpha(x)
=g(x_0),
\end{split}
\end{align}
by the reproducing property.
Hence, if $g_1\neq g_2$, then $\Lambda_{g_1}\neq\Lambda_{g_2}$.
We conclude that to each $\Lambda\in(\mathcal B^1_\alpha)^*$,
there corresponds a unique $g\in\mathcal B$ with
$\|g\|_{\mathcal B}\sim\|\Lambda\|$ and $\Lambda=\Lambda_g$.

We next show that $\mathcal B_0^*$ can be identified with
$\mathcal B^1_\alpha$ for any $\alpha>-1$.
For $f\in\mathcal B^1_\alpha$, define
$\Lambda_f\colon\mathcal B_0\to\mathbb C$ by
$\Lambda_f(g)=\langle f,g\rangle_\alpha$.
By \eqref{pairbound}, $\Lambda_f\in\mathcal B_0^*$ and
$\|\Lambda_f\|\lesssim\|f\|_{\mathcal B^1_\alpha}$.
Suppose now that $\Lambda\in\mathcal B_0^*$.
By Theorem \ref{Tproj}, $\Lambda\circ P_\alpha\in C_0(\mathbb B)^*$
and by the Riesz representation theorem, there exists a
complex Borel measure $\mu$ on $\mathbb B$ with
$\lvert\mu\rvert(\mathbb B)=\|\Lambda\circ P_\alpha\|$ such that
for all $\phi\in C_0(\mathbb B)$,
\begin{equation*}
(\Lambda\circ P_\alpha)\phi=\int_{\mathbb B}\phi(y)\,d\mu(y).
\end{equation*}
Pick a $t>0$.
For $g\in\mathcal B_0$, let
$\phi(x)=(1-\lvert x\rvert^2)^tD^t_\alpha g(x)$.
Then by Proposition \ref{PDstchar}, $\phi\in C_0(\mathbb B)$ with
$\|\phi\|_{L^\infty}\sim \|g\|_{\mathcal B}$ and
$P_\alpha\phi=g$.
Thus
\begin{equation*}
\Lambda(g)=(\Lambda\circ P_\alpha)\phi
=\int_{\mathbb B}\phi(y)\,d\mu(y)
=\int_{\mathbb B}(1-\lvert y\rvert^2)^tD^t_\alpha g(y)\,d\mu(y).
\end{equation*}
Further, since $D^t_\alpha g\in L^1_{\alpha+t}$, by the reproducing
property, $D^t_\alpha g=P_{\alpha+t}(D^t_\alpha g)$.
Inserting this into the above equation and using the symmetry of
$\mathcal R_{\alpha+t}$ we obtain
\begin{align}\label{dual2}
\begin{split}
\Lambda(g)&=\int_{\mathbb B}(1-\lvert y\rvert^2)^t
\int_{\mathbb B}\mathcal R_{\alpha+t}(y,x)D^t_\alpha g(x)
\,d\nu_{\alpha+t}(x)\,d\mu(y)\\
&=\int_{\mathbb B}
\int_{\mathbb B}\mathcal R_{\alpha+t}(x,y)(1-\lvert y\rvert^2)^t\,d\mu(y)
(1-\lvert x\rvert^2)^t D^t_\alpha g(x)\,d\nu_{\alpha}(x),
\end{split}
\end{align}
where we can change the order of the integrals since
$(1-\lvert x\rvert^2)^t D^t_\alpha g(x)$ is bounded, and
by Lemma \ref{LEstR}(a) and Lemma \ref{LIntonSandB},
$\int_{\mathbb B}\lvert \mathcal R_{\alpha+t}(y,x)\rvert\,d\nu_\alpha(x)
\lesssim (1-\lvert y\rvert^2)^{-t}$.
Let
\begin{equation*}
f(x):=\int_{\mathbb B}\mathcal R_{\alpha+t}(x,y)
(1-\lvert y\rvert^2)^t\,d\mu(y).
\end{equation*}
Then $f\in\mathcal H(\mathbb B)$ and by Fubini's theorem and
the estimate in the previous line,
\begin{equation*}
\|f\|_{L^1_\alpha}
\leq\int_{\mathbb B}(1-\lvert y\rvert^2)^t\int_{\mathbb B}
\lvert\mathcal R_{\alpha+t}(x,y)\rvert
\,d\nu_\alpha(x)\,d\lvert\mu\rvert(y)
\lesssim \lvert\mu\rvert(\mathbb B).
\end{equation*}
Thus $f\in\mathcal B^1_\alpha$ with
$\|f\|_{\mathcal B^1_\alpha}\lesssim\lvert\mu\rvert(\mathbb B)
=\|\Lambda\circ P_\alpha\|\lesssim\|\Lambda\|$, and $\Lambda=\Lambda_f$
by \eqref{dual2}.
Uniqueness of $f$ can be verified in the same way as the previous part.
We note that for $x_0\in\mathbb B$,
$\mathcal R_\alpha(x_0,\cdot)$ is in $\mathcal B_0$
because $\lvert\nabla\mathcal R_\alpha(x_0,\cdot)\rvert$
is bounded by Lemma \ref{LEstR}(b).
Using also the fact that $\mathcal R_{\alpha}(x_0,\cdot)$ is bounded,
we see that $f\mathcal R_{\alpha}(x_0,\cdot)$ is in $L^1_\alpha$ and
obtain
$\langle f,\mathcal R_{\alpha}(x_0,\cdot)\rangle_\alpha=f(x_0)$
as in \eqref{uniq}.
\end{proof}

We finish this section by verifying that there exists an unbounded
$\mathcal H$-harmonic Bloch function.

\begin{lemma}\label{Lunbdd}
There exists an unbounded function in $\mathcal B$.
\end{lemma}

\begin{proof}
Let $e_1=(1,0,\dots,0)\in\mathbb S$ and
$\phi(x)=(1-\lvert x\rvert^2)^{n-1}P_h(x,e_1)$, where $P_h$ is the
hyperbolic Poisson kernel in \eqref{Poisson}.
It is obvious that $\phi\in L^{\infty}(\mathbb B)$ and so the
Bergman projection
\begin{equation*}
f(x):=P_0\phi(x)
=\int_{\mathbb B}\mathcal R_0(x,y)P_h(y,e_1)
(1-\lvert y\rvert^2)^{n-1}\,d\nu(y)
\end{equation*}
is in $\mathcal B$ by Theorem \ref{Tproj}.
To see that $f$ is unbounded we find its series expansion.
Note that by the integral representation of $D^t_s$ in Lemma \ref{LDstInt}(ii)
(with $s=n-1$ and $t=-(n-1)$) we have $f(x)=D^{-(n-1)}_{n-1} P_h(x,e_1)$.
Therefore, by the series expansion of $P_h$ in \eqref{Poisson},
\begin{equation*}
f(x)=D^{-(n-1)}_{n-1} P_h(x,e_1)
=\sum_{m=0}^\infty\frac{c_m(0)}{c_m(n-1)}S_m(\lvert x\rvert)Z_m(x,e_1).
\end{equation*}
Observe that when $x=re_1$, all the terms in the above series are positive.
We have $Z_m(re_1,e_1)=r^mZ_m(e_1,e_1)\sim r^mm^{n-2}$ $(m\geq 1)$,
$c_m(0)/c_m(n-1)\sim m^{-(n-1)}$ by \eqref{cmasym}, and $S_m(r)\geq 1$
by \eqref{Smbound}.
Thus
\begin{equation*}
f(re_1)\gtrsim 1+\sum_{m=1}^\infty\frac{r^m}{m}
\end{equation*}
which tends to $\infty$ as $r\to 1^-$.
\end{proof}

\section{Atomic Decomposition}

Throughout the section we employ Proposition \ref{PDstchar}
and use any one of the equivalent norms
$\|(1-\lvert x\rvert^2)^t D^t_\alpha f(x)\|_{L^\infty}$
$(\alpha>-1, t>0)$ for the Bloch space $\mathcal B$.

\begin{lemma}\label{LkernelinB}
For every $\alpha>-1$ and $a\in\mathbb B$, the kernel
$\mathcal R_\alpha(\cdot,a)$ is in $\mathcal B_0$.
In addition, there exists $C=C(n,\alpha)>0$ such that
for all $a\in\mathbb B$,
\begin{equation}\label{NormR}
\frac{1}{C\,(1-\lvert a\rvert^2)^{\alpha+n}}
\leq \|\mathcal R_\alpha(\cdot,a)\|_{\mathcal B}
\leq \frac{C}{(1-\lvert a\rvert^2)^{\alpha+n}}.
\end{equation}
\end{lemma}

\begin{proof}
Pick some $t>0$.
For fixed $a\in\mathbb B$, by \eqref{DstRs},
$D^t_\alpha\mathcal R_\alpha(x,a)=\mathcal R_{\alpha+t}(x,a)$
which is bounded by Lemma \ref{LEstR}(a) and the
inequality $[x,a]\geq 1-\lvert a\rvert$ by \eqref{xybig}.
Thus $(1-\lvert x\rvert^2)^t D^t_\alpha \mathcal R_\alpha(x,a)$
is in $C_0(\mathbb B)$ and
$\mathcal R_\alpha(x,a)\in\mathcal B_0$.
Further,
\begin{equation*}
(1-\lvert x\rvert^2)^t\bigl\lvert\mathcal R_{\alpha+t}(x,a)\bigr\rvert
\lesssim \frac{(1-\lvert x\rvert^2)^t}{[x,a]^{\alpha+t+n}}
\lesssim\frac{1}{(1-\lvert a\rvert^2)^{\alpha+n}},
\end{equation*}
again by $[x,a]\geq 1-\lvert x\rvert$
and $[x,a]\geq 1-\lvert a\rvert$, which gives the second inequality
in \eqref{NormR}.
The first inequality follows from \cite[Lemma 6.1]{U1}
which shows that when $x=a$,
$\mathcal R_{\alpha+t}(a,a)\sim1/(1-\lvert a\rvert^2)^{\alpha+t+n}$,
and the fact that $\mathcal R_{\alpha+t}(\cdot,a)$ is continuous.
\end{proof}

\begin{lemma}\label{LOperT}
Suppose $\alpha>-1$ and $\{a_m\}$ is $r$-separated for some $0<r<1$.
Then the operator $T=T_{\{a_m\},\alpha}\colon\ell^\infty\to\mathcal B$
mapping $\lambda=\{\lambda_m\}$ to
\begin{equation*}
T\lambda(x)=\sum_{m=1}^\infty\lambda_m
\frac{\mathcal R_\alpha(x,a_m)}
{\|\mathcal R_\alpha(\cdot,a_m)\|_{\mathcal B}}
\qquad (x\in\mathbb B),
\end{equation*}
is bounded.
The above series converges absolutely and uniformly on
compact subsets of $\mathbb B$.
In addition, if $\lambda\in c_0$, then $T\lambda\in\mathcal B_0$.
\end{lemma}

\begin{proof}
We first verify that
$\sum_{m=1}^\infty (1-\lvert a_m\rvert^2)^{\alpha+n}<\infty$.
To see this, note that the balls $E_{r/2}(a_m)$ are disjoint,
and for fixed $r$,
$\nu(E_{r/2}(a_m))\sim (1-\lvert a_m\rvert^2)^n$ by \eqref{Ercenter}.
Also, for $y\in E_{r/2}(a_m)$, we have
$(1-\lvert y\rvert^2)\sim(1-\lvert a_m\rvert^2)$ by
Lemma \ref{Lratiobracket}(i).
Thus
\begin{equation*}
\sum_{m=1}^\infty (1-\lvert a_m\rvert^2)^{\alpha+n}
\sim\sum_{m=1}^\infty
\int_{E_{r/2}(a_m)}(1-\lvert y\rvert^2)^\alpha\,d\nu(y)
\leq\int_{\mathbb B}(1-\lvert y\rvert^2)^\alpha\,d\nu(y)
<\infty.
\end{equation*}
To see that the series absolutely and uniformly converges
on compact subsets of $\mathbb B$, suppose $\lvert x\rvert\leq R<1$.
Then $\lvert\mathcal R_\alpha(x,a_m)\rvert\leq C$ for all $m$ by
Lemma \ref{LEstR} and \eqref{xybig}.
Using also Lemma \ref{LkernelinB}, we obtain
\begin{equation*}
\sum_{m=1}^\infty\lvert\lambda_m\rvert
\frac{\lvert\mathcal R_\alpha(x,a_m)\rvert}
{\|\mathcal R_\alpha(\cdot,a_m)\|_{\mathcal B}}
\lesssim\|\lambda\|_{\ell^\infty}
\sum_{m=1}^\infty (1-\lvert a_m\rvert^2)^{\alpha+n}
<\infty.
\end{equation*}

Next, we pick some $t>0$.
By the continuity of $D^t_\alpha$ in Lemma \ref{LDstcont},
we can push $D^t_\alpha$ into the series and using
\eqref{DstRs}, Lemma \ref{LEstR}(a), Lemma \ref{LkernelinB},
and the fact that $[x,y]\sim[x,a_m]$ for $y\in E_{r/2}(a_m)$
by Lemma \ref{Lratiobracket}, we obtain
\begin{align*}
\big\lvert D^t_\alpha(T\lambda)(x)\big\rvert
&\lesssim\|\lambda\|_{\ell^\infty}\sum_{m=1}^\infty
\frac{(1-\lvert a_m\rvert^2)^{\alpha+n}}
{[x,a_m]^{\alpha+t+n}}
\sim\|\lambda\|_{\ell^\infty}\sum_{m=1}^\infty
\int_{E_{r/2}(a_m)}
\frac{(1-\lvert y\rvert^2)^\alpha}{[x,y]^{\alpha+t+n}}d\nu(y)\\
&\leq\|\lambda\|_{\ell^\infty}
\int_{\mathbb B}
\frac{(1-\lvert y\rvert^2)^\alpha}{[x,y]^{\alpha+t+n}}d\nu(y)
\lesssim\frac{\|\lambda\|_{\ell^\infty}}{(1-\lvert x\rvert^2)^t},
\end{align*}
where in the last inequality we use Lemma \ref{LIntonSandB}.
Hence $T\lambda\in\mathcal B$ and
$\|T\lambda\|_{\mathcal B}\lesssim\|\lambda\|_{\ell^\infty}$.

Finally, suppose $\lambda\in c_0$.
For $\varepsilon>0$, let $M$ be such that
$\sup_{m\geq M}\lvert\lambda_m\rvert<\varepsilon$.
Then
\begin{equation*}
T\lambda(x)
=\sum_{m=1}^{M-1}\lambda_m
\frac{\mathcal R_\alpha(x,a_m)}
{\|\mathcal R_\alpha(\cdot,a_m)\|_{\mathcal B}}
+\sum_{m=M}^\infty\lambda_m
\frac{\mathcal R_\alpha(x,a_m)}
{\|\mathcal R_\alpha(\cdot,a_m)\|_{\mathcal B}}
=:h_1(x)+h_2(x).
\end{equation*}
By Lemma \ref{LkernelinB}, $h_1$ is in $\mathcal B_0$,
and by the previous paragraph
$\|h_2\|_{\mathcal B}\lesssim\sup_{m\geq M}\lvert\lambda_m\rvert$.
Thus
$\limsup_{\lvert x\rvert\to 1^-}
(1-\lvert x\rvert^2)^t
\lvert D^t_\alpha(T\lambda)(x)\rvert\lesssim\varepsilon$
and $T\lambda$ is in $\mathcal B_0$.
\end{proof}

Following \cite{CR}, we associate with an $r$-lattice $\{a_m\}$
the following disjoint partition $\{E_m\}$ of $\mathbb B$.
Let $E_1=E_r(a_1)\backslash\bigcup_{m=2}^\infty E_{r/2}(a_m)$ and
for $m=2,3,\dots$, inductively define
\begin{equation*}
E_m=E_r(a_m)\backslash
\biggl(\bigcup_{k=1}^{m-1}E_k\,\bigcup\bigcup_{k=m+1}^\infty E_{r/2}(a_k)\biggr).
\end{equation*}
The following properties hold:
(i) $E_{r/2}(a_m)\subset E_m\subset E_r(a_m)$,
(ii) The sets $E_m$ are disjoint,
(iii) $\bigcup_{m=1}^\infty E_m=\mathbb B$.

\begin{lemma}\label{LOperU}
Suppose $\alpha>-1$, $t>0$, and $\{a_m\}$ is an $r$-lattice.
If $\{E_m\}$ is the associated sequence defined above, then
the operator
$U=U_{\{a_m\},\alpha,t}\colon\mathcal B\to\ell^\infty$,
\begin{equation*}
Uf=\Big\{D^t_\alpha f(a_m)\,\|\mathcal R_\alpha(\cdot,a_m)\|_{\mathcal B}
\,\nu_{\alpha+t}(E_m)\Big\}_{m=1}^\infty
\end{equation*}
is bounded.
In addition, if $f\in\mathcal B_0$, then $Uf\in c_0$.
\end{lemma}

\begin{proof}
Because $E_{r/2}(a_m)\subset E_m\subset E_r(a_m)$
and $r$ is fixed, by Lemma \ref{Lratiobracket} and \eqref{Ercenter},
$\nu_{\alpha+t}(E_m)
\sim(1-\lvert a_m\rvert^2)^{\alpha+t+n}$.
Combining this with Lemma \ref{LkernelinB} shows
\begin{equation*}
\lvert D^t_\alpha f(a_m)\rvert\,
\|\mathcal R_\alpha(\cdot,a_m)\|_{\mathcal B}\,\nu_{\alpha+t}(E_m)
\sim (1-\lvert a_m\rvert^2)^t\lvert D^t_\alpha f(a_m)\rvert.
\end{equation*}
Thus
$\|Uf\|_{\ell^\infty}
\lesssim\|f\|_{\mathcal B}$.
If $f\in\mathcal B_0$,
then $(1-\lvert x\rvert^2)^tD^t_\alpha f(x)\in C_0(\mathbb B)$
and so $Uf$ is in $c_0$, since $\lim_{m\to\infty}\lvert a_m\rvert=1$.
\end{proof}

\begin{proof}[Proof of Theorem \ref{TAtomic}]
Pick some $t>0$ and define the operators
$U\colon\mathcal B\to\ell^\infty$
and $T\colon\ell^\infty\to\mathcal B$ as above.
We show that there exists a constant $C=C(n,\alpha,t)$ such that
$\|I-TU\|_{\mathcal B\to\mathcal B}\leq Cr$, where $I$ is the identity operator.
This implies that $\|I-TU\|_{\mathcal B\to\mathcal B}<1$
when $r$ is sufficiently small, $TU$ is invertible,
and hence $T$ is onto.
In the little Bloch case replacing $\mathcal B$ with $\mathcal B_0$
and $\ell^\infty$ with $c_0$, we obtain
$T\colon c_0\to\mathcal B_0$ is onto.

Let $f\in\mathcal B$.
In the calculations below we suppress constants that depend only on
$n,\alpha,t$, and be careful that they do not depend on $r$ or $f$.
Note that
\begin{equation*}
TUf(x)=\sum_{m=1}^\infty
D^t_\alpha f(a_m)\mathcal R_\alpha(x,a_m)\,\nu_{\alpha+t}(E_m)
\end{equation*}
and the series converges absolutely and uniformly on compact subsets
of $\mathbb B$.
By continuity we can push $D^t_\alpha$ into the series,
and using \eqref{DstRs} obtain
\begin{align*}
D^t_\alpha (TUf)(x)
&=\sum_{m=1}^\infty
D^t_\alpha f(a_m)\mathcal R_{\alpha+t}(x,a_m)\,\nu_{\alpha+t}(E_m)\\
&=\sum_{m=1}^\infty\int_{E_m}
D^t_\alpha f(a_m)\mathcal R_{\alpha+t}(x,a_m)\,\nu_{\alpha+t}(y).
\end{align*}
Further, since $D^t_\alpha f\in L^1_{\alpha+t}(\mathbb B)$, by the
reproducing property,
\begin{equation*}
D^t_\alpha f(x)
=\int_{\mathbb B}\mathcal R_{\alpha+t}(x,y)D^t_\alpha f(y)d\nu_{\alpha+t}(y)
=\sum_{m=1}^\infty
\int_{E_m}\mathcal R_{\alpha+t}(x,y)D^t_\alpha f(y)d\nu_{\alpha+t}(y).
\end{equation*}
Thus
\begin{align*}
D^t_\alpha(I-TU)f(x)
&=\sum_{m=1}^\infty\int_{E_m}
\bigl(\mathcal R_{\alpha+t}(x,y)-\mathcal R_{\alpha+t}(x,a_m)\bigr)
D^t_\alpha f(y)\,d\nu_{\alpha+t}(y)\\
&+\sum_{m=1}^\infty\int_{E_m}
\bigl(D^t_\alpha f(y)-D^t_\alpha f(a_m)\bigr)\mathcal R_{\alpha+t}(x,a_m)
\,d\nu_{\alpha+t}(y)\\
&=:h_1(x)+h_2(x).
\end{align*}

We first estimate $h_1$.
Pick $y\in E_m$.
Since $E_m\subset E_r(a_m)$, a convex set, by the mean value inequality
\begin{equation*}
\bigl\lvert\mathcal R_{\alpha+t}(x,y)-\mathcal R_{\alpha+t}(x,a_m)\bigr\rvert
\leq \lvert y-a_m\rvert
\sup_{z\in E_r(a_m)}\lvert\nabla_z\mathcal R_{\alpha+t}(x,z)\rvert.
\end{equation*}
By \eqref{phmetric}, $\lvert y-a_m\rvert=\rho(y,a_m)[y,a_m]$ and since
$\rho(y,a_m)<r<1/2$, we have $[y,a_m]\sim [y,y]=1-\lvert y\rvert^2$ by
Lemma \ref{Lratiobracket}(ii), with the suppressed constants not
depending on $r$.
Thus $\lvert y-a_m\rvert\lesssim r(1-\lvert y\rvert^2)$.
Similarly, for all $x\in\mathbb B$ and $z\in E_r(a_m)$, we have
$[x,z]\sim[x,a_m]\sim[x,y]$ by Lemma \ref{Lratiobracket}(ii).
Hence, by Lemma \ref{LEstR}(b),
\begin{equation*}
\bigl\lvert\nabla_z\mathcal R_{\alpha+t}(x,z)\bigr\rvert
\lesssim\frac{1}{[x,z]^{\alpha+t+n+1}}
\sim\frac{1}{[x,y]^{\alpha+t+n+1}}.
\end{equation*}
We conclude that for all $x\in\mathbb B$ and $y\in E_m$,
\begin{equation}\label{KernelDiff}
\bigl\lvert\mathcal R_{\alpha+t}(x,y)-\mathcal R_{\alpha+t}(x,a_m)\bigr\rvert
\lesssim r\frac{1-\lvert y\rvert^2}{[x,y]^{\alpha+t+n+1}}
\lesssim \frac{r}{[x,y]^{\alpha+t+n}},
\end{equation}
where in the last inequality we use $[x,y]\geq 1-\lvert y\rvert$ by
\eqref{xybig}.
Thus
\begin{align*}
\lvert h_1(x)\rvert
&\lesssim r\sum_{m=1}^\infty\int_{E_m}
\frac{(1-\lvert y\rvert^2)^{\alpha+t}\lvert D^t_\alpha f(y)\rvert}{[x,y]^{\alpha+t+n}}
\,d\nu(y)
\lesssim r\|f\|_{\mathcal B}
\int_{\mathbb B}\frac{(1-\lvert y\rvert^2)^{\alpha}}{[x,y]^{\alpha+t+n}}
\,d\nu(y)\\
&\lesssim r\|f\|_{\mathcal B}\frac{1}{(1-\lvert x\rvert^2)^t},
\end{align*}
where the last inequality follows from Lemma \ref{LIntonSandB}.

To estimate $h_2$, note first that by the reproducing property, we have
\begin{equation*}
D^t_\alpha f(y)-D^t_\alpha f(a_m)
=\int_{\mathbb B}
\bigl(\mathcal R_{\alpha+t}(y,z)-\mathcal R_{\alpha+t}(a_m,z)\bigr)
D^t_\alpha f(z)\,d\nu_{\alpha+t}(z).
\end{equation*}
Therefore, by \eqref{KernelDiff} with the symmetry of $\mathcal R_{\alpha+t}$,
and Lemma \ref{LIntonSandB},
\begin{equation*}
\lvert D^t_\alpha f(y)-D^t_\alpha f(a_m)\rvert
\lesssim r\int_{\mathbb B}
\frac{(1-\lvert z\rvert^2)^{\alpha+t}\lvert D^t_\alpha f(z)\rvert}{[y,z]^{\alpha+t+n}}
\,d\nu(z)
\lesssim r\|f\|_{\mathcal B}\frac{1}{(1-\lvert y\rvert^2)^t}.
\end{equation*}
Hence, using also Lemma \ref{LEstR}, the fact that $[x,a_m]\sim [x,y]$
for all $x\in\mathbb B$ and $y\in E_m$, and finally Lemma \ref{LIntonSandB},
we deduce
\begin{align*}
\lvert h_2(x)\rvert
&\lesssim r\|f\|_{\mathcal B}
\sum_{m=1}^\infty\int_{E_m}
\frac{(1-\lvert y\rvert^2)^\alpha}{[x,a_m]^{\alpha+t+n}}\,d\nu(y)
\sim r\|f\|_{\mathcal B}
\sum_{m=1}^\infty\int_{E_m}
\frac{(1-\lvert y\rvert^2)^\alpha}{[x,y]^{\alpha+t+n}}\,d\nu(y)\\
&=r\|f\|_{\mathcal B}\int_{\mathbb B}
\frac{(1-\lvert y\rvert^2)^\alpha}{[x,y]^{\alpha+t+n}}\,d\nu(y)
\lesssim r\|f\|_{\mathcal B}\frac{1}{(1-\lvert x\rvert^2)^t}.
\end{align*}
Thus,
$\|(1-\lvert x\rvert^2)^t D^t_\alpha(I-TU)f(x)\|_{L^\infty}
\leq Cr\|f\|_{\mathcal B}$
and the proof is completed.
\end{proof}

To see that the representation \eqref{Repr2} can be used instead
of \eqref{atomicrepr}, the only change needed is to replace
$\|\mathcal R_\alpha(\cdot,a_m)\|_{\mathcal B}$ with
$(1-\lvert a_m\rvert^2)^{-(\alpha+n)}$ in the definitions of
$T$ and $U$.
The proofs of the Lemmas \ref{LOperT} and \ref{LOperU} become simpler;
and $TU$ and the proof of Theorem \ref{TAtomic} remain the same.


\begin{thebibliography}{10}
%
\bibitem{ABR}
S. Axler, P. Bourdon, W. Ramey,
Harmonic Function Theory,
2\textit{nd} ed., Grad. Texts in Math., vol. 137, Springer, New York, 2001.
%
\bibitem{CKL}
B. R. Choe, H. Koo, Y. J. Lee,
\textit{Positive Schatten class Toeplitz operators on the ball},
Studia Math. \textbf{189} (2008) 65--90.
%
\bibitem{CL}
B. R. Choe, Y. J. Lee,
\textit{Note on atomic decompositions of harmonic Bergman functions},
in Complex Analysis and its Applications, OCAMI Stud., vol. 2,
Osaka Munic. Univ. Press, Osaka, 2007, 11-24.
%
\bibitem{CR}
R. R. Coifman, R. Rochberg,
\textit{Representation theorems for holomorphic and harmonic functions in $L^p$},
Ast\'erisque \textbf{77} (1980) 12--66.
%
\bibitem{GKU}
S. Gerg\"{u}n, H. T. Kaptano\u glu, A. E. \"Ureyen,
Harmonic Besov spaces on the ball,
Int. J. Math. \textbf{27} (2016) no.9 1650070, 59 pp.
%
\bibitem{GJ}
S. Grellier, P. Jaming,
Harmonic functions on the real hyperbolic ball II.
Hardy-Sobolev and Lipschitz spaces,
Math. Nachr. \textbf{268} (2004) 50--73.
%
\bibitem{Ja1}
P. Jaming,
Trois probl\'emes d'analyse harmonique,
PhD thesis, Universit\'e d'Orl\'eans, 1998.
%
\bibitem{Ja2}
P. Jaming,
\textit{Harmonic functions on the real hyperbolic ball I.
Boundary values and atomic decomposition of Hardy spaces},
Colloq. Math. \textbf{80} (1999) 63--82.
%
\bibitem{JP1}
M. Jevti\'c, M. Pavlovi\'c,
Harmonic Bergman functions on the unit ball in $\mathbb R^n$,
Acta Math. Hungar. \textbf{85} (1999) 81--96.
%
\bibitem{JP2}
M. Jevti\'c, M. Pavlovi\'c,
Series expansion and reproducing kernels for hyperharmonic functions,
J. Math. Anal. Appl. \textbf{264} (2001) 673--681.
%
\bibitem{LS}
C. W. Liu, J. H. Shi,
\textit{Invariant mean-value property and $\mathcal M$-harmonicity in the
unit ball of $\mathbb R^n$},
Acta Math. Sin. \textbf{19} (2003) 187--200.
%
\bibitem{Min}
K. Minemura,
Harmonic functions on real hyperbolic spaces,
Hiroshima Math. J. \textbf{3} (1973) 121--151.
%
\bibitem{RK}
G. Ren, U. K\"{a}hler,
\textit{Pseudohyperbolic metric and uniformly discrete sequences
in the real unit ball},
Acta Math. Sci. \textbf{34B} no.~3 (2014) 629--638.
%
\bibitem{RKSL}
G. Ren, U. K\"{a}hler, J. Shi, C. Liu,
\textit{Hardy-Littlewood inequalities for fractional derivatives
of invariant harmonic functions},
Complex Anal. Oper. Theory \textbf{6} (2012) 373--396.
%
\bibitem{Sou}
M. P. Souza Pe\~{n}alosa,
Espacios de Bergman de funciones arm\'{o}nicas en la bola hiperb\'{o}lica,
Tesis de Doctorado, Posgrado en Ciencias Matem\'{a}ticas, Universidad
Nacional Aut\'{o}noma de M\'{e}xico, 2005.
%
\bibitem{St1}
M. Stoll,
Harmonic and Subharmonic Function Theory on the Hyperbolic Ball,
London Math. Soc. Lect. Note Series, vol. 431,
Cambridge University Press, Cambridge, 2016.
%
\bibitem{St2}
M. Stoll,
Reproducing kernels and radial eigenfunctions for the hyperbolic Laplacian,
preprint, available at https://www.researchgate.net/publication/304998931
%
\bibitem{St3}
M. Stoll,
The reproducing kernel of $\mathcal{H}^2$ and radial eigenfunctions
of the hyperbolic Laplacian,
Math. Scand. \textbf{124} (2019) 81--101.
%
\bibitem{U1}
A. E. \"Ureyen,
\textit{$\mathcal H$-Harmonic Bergman projection on the real hyperbolic ball},
J. Math. Anal. Appl. \textbf{519} (2023) 126802.
%
\bibitem{U2}
A. E. \"Ureyen,
\textit{Harmonic Bergman spaces on the real hyperbolic ball:
Atomic cecomposition, interpolation and inclusion relations},
preprint, available at https://arxiv.org/abs/2303.12137.
\end{thebibliography}
\end{document}